\documentclass[11pt, reqno]{amsart}

\usepackage{amsmath}
\usepackage{amssymb}
\usepackage{amsthm}
\usepackage{graphicx}
\usepackage{psfrag}
\usepackage{hyperref}

\allowdisplaybreaks

\newtheorem{thm}{Theorem}[section]
\newtheorem{lemma}[thm]{Lemma}
\newtheorem{prop}[thm]{Proposition}

\newtheorem{defn}[thm]{Definition}

\newtheorem{rmk}[thm]{Remark}

\newcommand{\mR}{\mathbb{R}}
\newcommand{\XX}{\mathcal{X}}

\newcommand{\dist}{\operatorname{dist}}

\def\R{{\mathbb R}}

\def\N{{\mathbb N}}

\def\B{{\mathcal B}}

\def\sm{\setminus}

\def\dist{\mbox{dist}}

\def\le{\leqslant}
\def\ge{\geqslant}

\newcommand{\E}{\mathcal{E}}

\newcommand{\beq}{\begin{equation}}
\newcommand{\eeq}{\end{equation}}
\def\n{\ensuremath{n}}

\def\E{\mathbb E}

\def\tDr{\tilde{\mathcal{D}}_r}

\title{Weak convergence to extremal processes and record events for 
non-uniformly hyperbolic dynamical systems}
\author[M. Holland]{Mark Holland}
\address{Mark Holland\\
Mathematics (CEMPS)\\ 
Harrison Building (327)\\
North Park Road\\ 
Exeter\\
EX4 4QF\\
UK}
\email{\href{mailto:M.P.Holland@exeter.ac.uk}{M.P.Holland@exeter.ac.uk}}
\urladdr{\url{http://empslocal.ex.ac.uk/people/staff/mph204/}}

\author[M. Todd]{Mike Todd}
\address{Mike Todd\\ Mathematical Institute\\
University of St Andrews\\
North Haugh\\
St Andrews\\
KY16 9SS\\
Scotland} \email{\href{mailto:m.todd@st-andrews.ac.uk}{m.todd@st-andrews.ac.uk}}
\urladdr{\url{http://www.mcs.st-and.ac.uk/~miket/}}

\date{\today}

\begin{document}

\maketitle

\begin{abstract}
For a measure preserving dynamical system $(\XX,f, \mu)$, we consider the time series of maxima 
$M_n=\max\{X_1,\ldots,X_n\}$ associated to the process $X_n=\phi(f^{n-1}(x))$ generated by the
dynamical system for some observable $\phi:\XX\to\mathbb{R}$. Using a point process approach we establish
weak convergence of the process $Y_n(t)=a_n(M_{[nt]}-b_n)$ to an extremal process $Y(t)$ for suitable scaling
constants $a_n,b_n\in\mathbb{R}$. Convergence here taking place in
the Skorokhod space $\mathbb{D}(0,\infty)$ with
the $J_1$ topology. We also establish distributional results for the record 
times and record values of the corresponding maxima process.
\end{abstract}

\section{Introduction}
Consider a measure preserving dynamical system $(\XX, f, \mu)$, where $\XX\subset\mathbb{R}^d$, 
$f:\XX\rightarrow \XX$ is a measurable transformation, and $\mu$ is an $f$-invariant probability measure 
supported on $\XX$. Given a measurable (observable) function
$\phi:\XX \rightarrow \mR$ we consider the stationary stochastic process
$X_1, X_2, \dots$ defined as
\begin{equation}\label{eq:phi_o_f} 
X_i =\phi \circ f^{i-1}, \quad i \geq 1,  
\end{equation}
and its associated maximum process $M_n$ defined as
\begin{equation}
\label{eq:max-process} 
M_n = \max(X_1,\dots,X_n). 
\end{equation}
Under appropriate assumptions on the system $(\XX, f, \mu)$, there exist scaling constants
$a_n>0$ and $b_n\in\mathbb{R}$ and a non-degenerate limit function $G(u)$ for which
\begin{equation}\label{eq.Mn-limit}
\lim_{n\to\infty}\mu\{x\in\XX:\, a_n(M_n(x)-b_n)\leq u\}=G(u). 
\end{equation}
Beyond the distributional limit established in \eqref{eq.Mn-limit}, here
we consider the continuous time 
process $\{Y_n(t):t\geq 0\}$ defined by
\begin{equation}\label{eq.Y-process}
Y_n(t)=
\begin{cases}a_n(M_{[nt]}-b_n) & t\geq 1/n;\\
a_n(X_1-b_n) & 0<t<1/n.
\end{cases} 
\end{equation}
For each $n\geq 1$, $Y_n(t)$ is a random graph with values in the Skorokhod space $\mathbb{D}(0,\infty)$. Under suitable
hypotheses on $(\XX, f, \mu)$ we prove existence of a non-degenerate limit process $Y(t)$ so that $Y_n(t)\overset{{\small d}}{\to}Y(t)$ in $\mathbb{D}(0,\infty)$ with respect to the Skorokhod $J_1$ metric. Here, $\overset{{\small d}}{\to}$ denotes weak convergence (or convergence in distribution). 
The limit process $Y(t)$ will be the so called \emph{extremal process} which we now define.

\subsection{Extremal processes and weak convergence.}
Consider a general probability space $(\Omega,\mathcal{B},\mu)$, where $\mathcal{B}$
is the $\sigma$-algebra of sets in the sample space $\Omega$. If $X:\Omega\to\mathbb{R}$ is a random variable, we let
$F(u):=\mu\{X\leq u\}$, and define finite dimensional distributions:
\begin{equation}
F_{t_1,\ldots,t_k}\left(u_1,\ldots,u_k\right)=
F^{t_1}\left(\bigwedge_{i=1}^{k}\{u_i\}\right)
F^{t_2-t_1}\left(\bigwedge_{i=2}^{k}\{u_i\}\right)\cdots 
F^{t_k-t_{k-1}}(u_k),
\end{equation} 
with $t_1<t_2<\cdots<t_k$, and $\wedge$ denoting the minimum operation. Suppose that
$Y_{F}(t)$ is a stochastic process with these finite dimensional distributions, i.e.
\begin{equation}
\mu\{Y_F(t_1)\leq u_1,\ldots Y_F(t_k)\leq u_k\}=F_{t_1,\ldots,t_k}(u_1,\ldots,u_k).
\end{equation}
By the Kolmogorov extension theorem such a process exists and is called an \emph{extremal-$F$ process}. A version
can be taken in $\mathbb{D}(0,\infty)$, i.e. continuous to the right with left hand limits. 
It turns out that $Y_F(t)$ is a Markov jump process, see \cite{Embrechts, Res_weak, Resnick3}, and further properties include:
\begin{itemize}
\item For all $t,s>0$, $\mu\{Y_F(t+s)\leq x\mid Y_F(s)=y\}= F^t(x)\cdot 1_{\{x\geq y\}}$.
\item For all $t,s>0$, $\mu\{Y_F(t+s)=Y_F(t)\mid Y_F(s)=y\}= F^t(y)$. Setting $Q(y)=-\log F(y)$ implies that the holding
time in state $y$ is given by an exponential distribution with parameter $Q(y)$.
\item If $\{t_i\}$ denotes the sequence of jump times (or points of discontinuity) for $Y_F(t)$ then
\begin{equation}
\mu\{Y_F(t_{i+1})\leq x\mid Y_F(t_i)=y\}=
\begin{cases}
1-Q(x)/Q(y) & x>y;\\
0 & \text{if $x<y$}.
\end{cases}
\end{equation}
\end{itemize}

Our main result is to show that for certain chaotic dynamical systems the process $Y_n(t)$ in \eqref{eq.Y-process} converges (weakly) to an extremal-$G$ process $Y_G(t)$. This is the first time extremal processes have been used in the dynamical systems context.  The mode of convergence to the extremal process is in distribution on $\mathbb{D}(0,\infty)$ with respect to the Skorokhod $J_1$-topology. To be precise, let $\mu_n=\mu\circ Y^{-1}_{n}(t)$,
so that for all measurable $A\subset\mathbb{R}$, $Y^{-1}_n(t)(A)=\{\omega\in\Omega:Y_n(t)(\omega)\in A\}$.
Then $Y_n(t)\overset{{\small d}}{\to}Y(t)$ in $\mathbb{D}(0,\infty)$ if for all bounded
continuous functions $\varphi$
on $\mathbb{D}(0,\infty)$,
$$\int_{(0,\infty)}\varphi(x)d\mu_n\to\int_{(0,\infty)}\varphi(x)d\mu.$$
The space $\mathbb{D}(0,\infty)$ consists of right continuous functions, with existence of limits to the left
(cadlag functions) \cite{Skorokhod}. To overview the construction of the Skorokhod $J_1$-topology consider first the space $\mathbb{D}[a,b]$, and let
$\|\cdot\|$ denote the uniform norm on $[a,b]$, so that $\|\varphi\|=\sup_{a\leq t\leq b}|\varphi(t)|$. Then
a metric on $\mathbb{D}[a,b]$ is given by:
$$d_{a,b}(\varphi_1,\varphi_2):=\inf_{h\in\Lambda}\left\{\|\varphi_1\circ h-\varphi_2\|\vee\|h-\mathrm{id}\|\right\},$$
where $\Lambda$ is the set of strictly increasing, continuous functions
$h:[a,b]\to[a,b]$ such that $h(a)=a$ and $h(b)=b$. The function $\mathrm{id}$ is the identity mapping.
This metric is not complete, 
but an equivalent (complete) metric can be constructed by a homeomorphism,
see \cite{Resnick3, Skorokhod}. 

The construction carries over to $\mathbb{D}(0,\infty)$ by use of the following metric: let $r_{a,b}\varphi(x)$ denote
the restriction of $\varphi(x)$ to the interval $[a,b]$, and define 
$$d_{0,\infty}(\varphi_1,\varphi_2):=\int_{0}^{1}\int_{1}^{\infty}e^{-t}(1\wedge d_{s,t}(r_{s,t}\varphi_1,
r_{s,t}\varphi_2))\,dt\,ds.$$
Then convergence $\varphi_n\to\varphi$ in $\mathbb{D}(0,\infty)$ holds in the $J_1$ metric if 
$d_{0,\infty}(\varphi_n,\varphi)\to 0$ at each continuity point of $\varphi$. We remark that similar metrics can be constructed on $\mathbb{D}(-\infty,0)$ and $\mathbb{D}(-\infty,\infty)$.

\section{Main results: weak convergence to an extremal process.}
As noted above, the heart of this paper is to prove that the process $Y_n$ converges to a certain extremal process.  In this section we give the main results in this direction after defining the relevant short-term and long-term mixing conditions that guarantee this convergence.  In the next section these results will then be interpreted in terms of record times and record values via the Continuous Mapping Theorem.

We consider $(\XX, f, \mu)$ a measure preserving system, 
and assume that $\mu$ is absolutely continuous with respect to Lebesgue $m$ 
with density $\rho$ (note that this condition could be removed in line with \cite{FFT_rough}). Within this article, unless otherwise stated we
consider observable functions of the form
$\phi(x)=\psi(\dist(x,\tilde{x}))$. Here $\tilde{x}$ is a chosen
point in $\XX$, and $\psi:[0,\infty)\to\mathbb{R}$ is a measurable
function with $\sup_{v\in[0,\infty)}\psi(v)=\psi(0)$.

\subsection{Probabilistic mixing and recurrence conditions}
Let $\mathcal{S}$ denote the semi-ring of subsets whose elements are intervals of type $[a,b)$
for $a,b\in\mathbb{R}^+$, and let $\mathfrak{R}$ denote the ring generated by $\mathcal{S}$.
So for every $A\in\mathfrak{R}$, there exists $k\in\mathbb{N}$, and intervals $I_1,\ldots, I_k\in\mathcal{S}$
such that $A=\cup_j I_j$. To fix notation, for any $\alpha\in\mathbb{R}$ and $I=[a,b)\in\mathcal{S}$
we have $\alpha I=[\alpha a,\alpha b)$, and $I+\alpha=[a+\alpha,b+\alpha)$. This notation extends 
in a natural way to any $A\in\mathfrak{R}$. For $A\in\mathfrak{R}$, we let
$$M(A)=\max\{X_i,i\in A\cap\mathbb{Z}\}.$$
In the case $A=[0,n)$, we have $M(A)=M_n$. 

Given this setup, the probabilistic condition we define first concerns that of mixing and asymptotic independence of maxima in different blocks. We make this precise as follows. Given $0<x_1\leq x_2\leq\cdots\leq x_r$, suppose that $u^{(1)}_{n}\geq u^{(2)}_{n}\geq\cdots\geq u^{(r)}_{n}$
are such that 
\begin{equation}\label{eq.u-kn}
n\mu\{X_1>u_{n}^{(k)}\}\to x_k,\quad\forall k\leq r.
\end{equation}
\begin{defn}
We say that condition ${\mathcal{D}}_r(u_n^{(k)})$ holds, if for any disjoint collection of sets 
$\mathcal{A}_1,\ldots, \mathcal{A}_r\in \mathfrak{R}:$
\begin{multline*}%\label{eq.dr_un}
\left|\mu\left(\{X_1>u_{n,0}\}\cap\{M(\mathcal{A}_1+t)\leq u_{n,1},\ldots, 
M(\mathcal{A}_r+t)\leq u_{n,r}\}\right)\right.\\
\left. -\mu\left(\{X_1>u_{n,0}\}\right)\mu\left(\{M(\mathcal{A}_1)\leq u_{n,1},\ldots, 
M(\mathcal{A}_r)\leq u_{n,r}\}\right)\right|
\leq\gamma(n,t),
\end{multline*}
where $n\gamma(n,t)\to 0$ for some integer sequence $t_n=o(n)$, and for each $i$, 
$u_{n,i}$ denotes any one of the 
$u^{(k)}_{n}$, $(1\leq k\leq r)$.
\end{defn}
In the next definition, we consider the frequency of exceedances of the $X_j$
(in a probabilistic sense) over a threshold sequence $\{u_n\}$.
\begin{defn}
We say that condition $\mathcal{D}_r'(u^{(k)}_n)$ holds for a sequence $\{u^{(k)}_n\}$ if
\begin{equation}
\lim_{k\to\infty}\limsup_{n\to\infty} n\sum_{j=2}^{n/k}\mu\left(\left\{X_1>u^{(k)}_n,X_j>u^{(k)}_n\right\}\right)=0.
\end{equation} 
\end{defn}
In the case where it is known that %$\mu(\{M_{n}\leq u/a_n+b_n\})\to G(u)$ 
\eqref{eq.Mn-limit} holds for 
given sequences $\{a_n\}, \{b_n\}$ and non-degenerate distribution function $G(u)$ 
we can usually take 
\begin{equation}\label{eq.un_rep}
u^{(k)}_n( x_k)=a_n^{-1}G^{-1}(e^{ -x_k})+b_n
\end{equation}
in the 
above definitions. This is certainly true in the i.i.d case, and
for a wide class of dynamical systems. In fact when this scaling rule applies it is more natural
to re-write \eqref{eq.u-kn} and assume the representation
\begin{equation}\label{eq.G_rep}
n\mu\{X_1>u/a_n+b_n\}\to-\log G(u),
\end{equation}
i.e. to replace $x_k$ by $-\log G(u)$ for some $u\in\mathbb{R}$. The limit relation \eqref{eq.Mn-limit} becomes
a consequence of \eqref{eq.G_rep} (rather than the converse) and is known as the \emph{Poisson approximation}, 
see \cite{Embrechts}.

We briefly compare our conditions with previous ones.  In the classical probability literature there were conditions 
$D$ and $D'$, see \cite{Leadbetter}, but Freitas and Freitas \cite{FF, FFT1}, inspired by Collet \cite{Collet} converted 
these to conditions more straightforwardly checkable in a dynamical context. In conjunction with \eqref{eq.G_rep} these conditions
are then used to imply the limit law in \eqref{eq.Mn-limit}.  Condition $\mathcal{D}_r(u^{(k)}_n)$ is very similar to $D_3(u_n)$ used in \cite[Section 4]{FFT1} (itself similar to $D(u_n)$ in that paper): the only difference being the multiple thresholds $u_n^{(k)}$ for each given $n$ leading to a slightly stronger condition on the mixing.  Also $\mathcal{D}_r'(u^{(k)}_n)$ is nearly identical to $D'(u_n)$ in \cite{FFT1}, but in this case the different thresholds do not add any real strength to the condition since they are each checked independently.  The similarities between the conditions here and in \cite{FFT1} are strong enough that by examining the proofs (most importantly of the $D$ condition), any dynamical system to which the old conditions have been shown to apply can quite easily be seen to satisfy our new conditions.

\begin{thm}\label{thm.Y-process}
Suppose that $(\XX, f, \mu)$ is a measure preserving dynamical system, 
and suppose that there exist sequences $\{a_n\}, \{b_n\}$ and a non-degenerate
function $G(u)$ such that \eqref{eq.G_rep} holds.  
Suppose that for any $r\geq 1$, and any sequence $\{ x_i,i\leq r\}$, condition 
$\mathcal{D}_r(u^{(k)}_n)$ holds together with $\mathcal{D}_r'(u^{(k)}_n)$. 
Then the process $Y_n(t)=a_n(M_{[nt]}-b_n)$ converges weakly to an 
extremal-$G$ process $Y_G(t)$ in $\mathbb{D}(0,\infty)$ (endowed with 
the Skorokhod $J_1$ topology). 
\end{thm}

We make several remarks about Theorem \ref{thm.Y-process}. 
First note that the hypotheses on $(\XX, f, \mu)$ as stated are quite weak.
However to check conditions $\mathcal{D}_r(u^{(k)}_n)$ and $\mathcal{D}_r'(u^{(k)}_n)$
in specific applications we usually require absolute continuity of $\mu$ and estimates on 
the recurrence time statistics (e.g. ergodicity and/or decay of correlation estimate). 
We discuss this in Section \ref{sec.applications}. A further remark is that conditions 
$\mathcal{D}_r(u^{(k)}_n)$ and $\mathcal{D}_r'(u^{(k)}_n)$ are sufficient conditions to ensure
convergence to $Y_G(t)$. In the i.i.d case, existence of the limit \eqref{eq.Mn-limit} alone is enough
to ensure convergence, see \cite{Resnick3}. Thus in certain situations we might expect the conclusion of 
Theorem \ref{thm.Y-process} to hold under weaker hypotheses. We discuss this further in Section \ref{sec.discussion}.

Equation \eqref{eq.G_rep} in itself imposes regularity conditions
on the invariant density of $\mu$ and on the form of the observable 
$\phi(x)=\psi(\dist(x,\tilde{x}))$. For linear scaling sequences $\{a_n\}, \{b_n\}$ there turn out to be three non-degenerate 
types for the distribution function $G$, namely Gumbel, Fr\'echet and Weibull \cite{Embrechts}. We make the following remark detailing how these types can arise, and for the precise computations see e.g. \cite{HNT}.
\begin{rmk}%\label{cor.Y-process}
Suppose that $(\XX, f, \mu)$ is a measure preserving system, for given $\tilde{x}\in\XX$ the invariant density $\rho(\tilde{x})$ lies in $(0, \infty)$ and for any $r\geq 1$, and any sequence $\{ x_i,i\leq r\}$ 
condition $\mathcal{D}_r(u^{(k)}_n)$ holds together with $\mathcal{D}_r'(u^{(k)}_n)$. We have the following cases. 
\begin{itemize}
\item[(i)] If $X_n=-\log(\dist(f^{n-1}(x),\tilde{x}))$, then $Y_n(t)=M_{[nt]}-\log n$ converges weakly to the extremal-$G$ 
process $Y_G(t)$ in $\mathbb{D}(0,\infty)$. In this case 
$G(u)=e^{-2\rho(\tilde{x})e^{-u}}$.
\item[(ii)] If $X_n=\dist(f^{n-1}(x),\tilde{x})^{-\alpha}$ for $\alpha>0$, then $Y_n(t)=n^{-\alpha}M_{[nt]}$ converges weakly to the extremal-$G$ 
process $Y_G(t)$ in $\mathbb{D}(0,\infty)$. In this case 
$G(u)=e^{-2\rho(\tilde{x})u^{-\alpha}}$.
\item[(iii)] 
If $X_n=C-\dist(f^{n-1}x,\tilde{x})^{\alpha}$ for $\alpha>0$, then the process $Y_n(t)=n^{\alpha}(M_{[nt]}-C)$ 
converges weakly to the extremal-$G$ process $Y_G(t)$ 
in $\mathbb{D}(0,\infty)$. In this case $G(u)=e^{-2\rho(\tilde{x})(-u)^{\alpha}}$.
\end{itemize}
In each case, $\mathbb{D}(0,\infty)$ is endowed with the Skorokhod $J_1$ topology.
\end{rmk}
Given the extremal process $Y_G(t)$ its (path) inverse is defined by:
$$Y^{\leftarrow}_G(t)=\inf\{x:Y(x)>t\},$$
where the domain of $Y^{\leftarrow}_G$ is the left and right end-points of $G$.
Given $Y_n(t)=a_n(M_{[nt]}-b_n)$,  denote the inverse path process by 
$Y^{\leftarrow}_n(t)$. We have the following result.
\begin{thm}\label{thm.Y-inv.process}
Suppose that $(\XX, f, \mu)$ is a measure preserving dynamical system,
and suppose that there exist sequences $\{a_n\}, \{b_n\}$ and a non-degenerate
function $G(u)$ such that \eqref{eq.G_rep} holds.
Suppose that for any $r\geq 1$, and any sequence $\{ x_i,i\leq r\}$ condition 
$\mathcal{D}_r(u^{(k)}_n)$ holds together with $\mathcal{D}_r'(u^{(k)}_n)$.  Then 
the process $Y^{\leftarrow}_n(x)$ converges weakly to the inverse extremal-$G$ 
process $Y^{\leftarrow}_G(t)$ in $\mathbb{D}(E)$, (endowed with 
the Skorokhod $J_1$ topology), where
$E\subset\mathbb{R}$ is the domain of definition of $G$.
\end{thm} 
For example, in Theorem \ref{thm.Y-inv.process} we take $E=(-\infty,\infty)$
in the case of the Gumbel distribution. For the Fr\'echet and Weibull distributions
we take $E=(0,\infty)$ and $E=(-\infty,0)$ respectively.

\section{The distribution of record times and record values}
Given the processes $Y_G(t)$ and $Y_{G}^{\leftarrow}(t)$, we describe
next the distribution of their jump values, i.e. the locations of
their discontinuities. This has natural application to the theory of record times and
record values which we describe as follows. Consider the original processes $\{X_n\}$, $\{M_n\}$, and 
let $\tau_1=1$. Define a strictly increasing sequence $\{\tau_k\}$ via:  
\begin{equation}\label{eq.time-to-record}
\tau_k=\inf\{j>\tau_{k-1}:X_j>M_{j-1}\}.
\end{equation}
Then this sequence $\{\tau_k\}$ forms the \emph{record times} associated
to the process $M_n$, namely the times where $M_n$ jumps. The corresponding record values are given 
by the $X_{\tau_k}=M_{\tau_k}$. For the process $Y_n(t)=a_n(M_{[nt]}-b_n)$, we see
that the jumps of $Y_n(t)$ occur precisely at the times $t_k=\tau_k/n$ where $\tau_k$ is
a record time. The jump values $Y_n(t_k)$ are then the (normalised) record values
$a_n(X_{\tau_k}-b_n).$ We shall use point process theory to describe the distributional
behaviour of these jump times and jump values.

\subsection{Overview of point process theory}
To study extremal processes and their corresponding jump processes we use
a point process approach. We recall some general properties of point 
processes, see \cite[Chapter 3]{Resnick3}.
Let $(\Omega,\mathcal{B},\mu)$ be a probability space,
where $\mathcal{B}$ is a $\sigma$-algebra of subsets of $\Omega$.
Let $E\subset\mathbb{R}^d$ be a state space, with Borel $\sigma$-algebra
$\mathcal{E}$. The set $E$ is the region in which points will be defined.
Given a sequence of (vector-valued) random variables $X_i:\Omega\to E$,
consider the quantity 
$\xi=\sum_{i=1}^{\infty}\delta_{X_i}$, where $\delta_x$ is the 
Dirac measure at $x$. 
Then $\xi:(\Omega,\mathcal{B},\mu)\to (M_{p}(E),
\mathcal{M}_p(E))$ defines a point process. The set $M_{p}(E)$ is
the collection of point (counting) measures $m$ on $E$, with
$m(A)<\infty$ if $A\subset E$ is compact.
The set $\mathcal{M}_p(E)$ is the corresponding $\sigma$-algebra of subsets 
of $M_p(E)$.
Thus given $A\subset E$, $\xi(A)\in\mathbb{N}$ is itself a random variable 
associated to the sequence $\{X_i\}$, i.e., given $\omega\in\Omega$, 
and $A\subset E$,
$$\xi:=\xi(A,\omega)=\sum_{i=1}^{\infty}\delta_{X_i(\omega)}(A)
=\sharp\{i:X_i(\omega)\in A\}.$$

Of interest to us are the special class of point processes known
as \emph{Poisson random measures} (PRM). We say that
a point process $\xi$ is a PRM with mean measure $\lambda$ 
if the following hold:
\begin{enumerate}
\item For all $A\in\mathcal{E}$,
$$\mu(\xi(A, \cdot)=k)=\frac{\lambda(A)^ke^{-\lambda(A)}}{k!}.$$%,\quad\forall\,A\subset\mathbb{R}.$$
\item If $A_i\cap A_j=\emptyset$ for all $i,j\in\{1,\ldots,m\}$, and all
$m\geq 1$ then $\xi(A_1),\ldots,\xi(A_m)$ are independent random variables.
\end{enumerate}
The mean (or \emph{intensity}) measure $\lambda$ satisfies 
$\lambda(A)=\E(\xi(A))$ for all $A\in\mathcal{E}$. If we can write
$\lambda=\int_{A}\gamma(t)\,dt$, then we call $\gamma(t)$ 
the intensity of the process $\xi$. 

We now consider convergence of a sequence of point processes. 
A sequence $\{\xi_n\}$ of point processes converges in distribution (or converges weakly) to a point 
process $\xi$ in $M_p(E)$ if for any finite collection
$B_1,\ldots,B_m$ of bounded Borel sets in $E$ with $\mu(\xi(\partial B_i)=0)=1$, 
$$(\xi_n(B_1),\xi_n(B_2),\ldots, \xi_n(B_m))
\to (\xi(B_1),\xi(B_2),\ldots, \xi(B_m)).$$
To prove such convergence for sequences of (simple) point processes it suffices
to check the following criteria due to Kallenberg \cite{Kallenberg}:
\begin{enumerate}
\item[(a)] $\E(\xi_n(B))\to \E(\xi(B))$, where 
$B$ is a semi-closed rectangle in $E$.
\item[(b)] $\mu\{\xi_n(B)=0\}\to\mu\{\xi(B)=0\}$ for all 
finite unions of semi-closed rectangles $B\subset E$.
\end{enumerate}
In the definition above, we take semi-closed (and disjoint unions) of intervals of the form
$(a,b]$ for subsets of $\mathbb{R}$, while in $\mathbb{R}^2$ we 
take semi-closed (and disjoint unions of) rectangles of the form $(a,b]\times (c,d]$.

Next we discuss vague convergence of measures, and convergence of certain transformations
of measures. These concepts will be used in Section \ref{sec.extreme_process}.
We say that a sequence of measures $\{\mu_n\}$ convergences vaguely to $\mu$
(written $\mu_n\overset{\small{v}}{\to}\mu$), if $\int g\, d\mu_n\to\int g\,d\mu$ for all continunuous 
functions $g$ that are compactly supported. 
In general vague convergence does not imply weak convergence, e.g. $\mu_n=\delta_n$ (the Dirac mass
at $n\in\mathbb{N}$) converges vaguely to $\mu=0$, but not weakly. For point processes,
there exists a vague metric which makes $M_p(E)$ a complete metric space.

To study weak convergence of certain functionals of point process we will
make use of the \emph{Continuous Mapping Theorem} (CMT).
Consider metric spaces $M$ and $M'$ and let 
$\mu$ be a probability measure. 
A function $h: M\to M'$ is a.s. continuous if the set of 
discontinuities of $h$ has $\mu$-measure zero. Suppose that $\xi_n\overset{\small{d}}{\to} \xi$ in $M$,
then the Continuous Mapping Theorem asserts that
$h(\xi_n)\overset{\small{d}}{\to} h(\xi)$ in $M'$ 
provided $h$ is a.s.\ continuous.

\subsection{The record time and record value point processes }
We now state distributional results for the record-time and record-value
jump processes. For the process $Y_n(t)=a_n(M_{[nt]}-b_n)$, recall that
$t_k=\tau_k/n$ are the jump times. i.e. where $X_{\tau_k}>M_{\tau_k-1}$.
The jump values are given by $Y(t_k)$. We consider the following
two point processes (defined on subsets of $\mathbb{R}$):
\begin{equation}
\mathcal{R}_n:=\sum_{j=1}^{\infty}\delta_{\frac{j}{n}}\cdot 1_{\{X_j>M_{j-1}\}},\quad
\mathcal{V}_n:=\sum_{\tau_k}\delta_{Y_n(\tau_k/n)},
\end{equation}
the former is the \emph{record time process} and the latter is the \emph{record value process}.
\begin{thm}\label{thm.records}
Suppose that $(\XX, f, \mu)$ is a measure preserving dynamical system,
and suppose that there exist sequences $\{a_n\}, \{b_n\}$ and a non-degenerate
function $G(u)$ such that \eqref{eq.G_rep} holds.
Suppose that for any $r\geq 1$ and any sequence $\{ x_i,i\leq r\}$ condition 
$\mathcal{D}_r(u^{(k)}_n)$ holds together with $\mathcal{D}_r'(u^{(k)}_n)$. We have the following
cases:
\begin{enumerate}
\item The point process $\mathcal{R}_n$ converges weakly to the point process $\mathcal{R}$
on state space $(0,\infty)$. The process $\mathcal{R}$ is a PRM with intensity $\gamma(t)=1/t$.
i.e. for any $0<a<b<\infty$:
$$\lim_{n\to\infty}\mu\{\mathcal{R}_n(a,b)=k\}= \frac{(\log(b/a))^k}{k!}\cdot\frac{a}{b}.$$
\item The point process $\mathcal{V}_n=\sum_{\tau_k}\delta_{Y(\tau_k/n)}$
converges weakly to the point process $\mathcal{V}$ on state space $E\subset\mathbb{R}$
(the domain of $G$), where $\mathcal{V}$ is a PRM with intensity measure 
$\lambda_{V}$ given by $\lambda_V([a,b])=-\log(-\log G(b))+\log(-\log G(a))$.
\end{enumerate}
\end{thm}
Hence Theorem \ref{thm.records} implies that the process $\mathcal{R}_n$
converges to a PRM $\mathcal{R}$ on $E=(0,\infty)$ with intensity $\gamma(t)=1/t$
irrespective of the underlying distribution function $G$. We remark that if item (1) of the Theorem \ref{thm.records} holds
for a particular observable $\phi:\XX\to\mathbb{R}$, then $\mathcal{R}_n\overset{{\small d}}{\to}\mathcal{R}$ 
holds for any injective 
and monotone increasing transformation of $X_1$.
By contrast, the limit process $\mathcal{V}_n$ does depend on $G$, and hence on the form of $\phi$. We do not consider all possibilities,
but remark that in the case of $G$ being the Gumbel distribution,
$$\lim_{n\to\infty}\mu\{\mathcal{V}_n(a,b)=k\}= \frac{(b-a)^ke^{-(b-a)}}{k!}.$$
This would apply to the dynamical process $X_n=-\log\dist(f^{n-1}(x),\tilde{x})$.
We remark further that
the process $\mathcal{V}_n$ determines the jump \emph{times}
for inverse process  $Y_{G}^{\leftarrow}(t)$.

\section{Application of Results}\label{sec.applications}

In this section we give an overview of dynamical system models
that fit within our assumptions, namely $\mathcal{D}_r(u^{(k)}_n)$ and
$\mathcal{D}_r'(u^{(k)}_n)$.  As mentioned above, conditions of this type have been considered in many recent papers.  
In most applications it is the short-range condition which is hardest to prove, and as noted above, once the usual short range condition is proved, $\mathcal{D}_r'(u^{(k)}_n)$ follows immediately.  Checking that the long-range mixing condition 
$\mathcal{D}_r(u^{(k)}_n)$ holds follows almost exactly as usual, namely via decay of correlations:

We say that $(\XX, f, \nu)$ has decay of correlations in Banach spaces
$\mathcal{B}_1$ versus $\mathcal{B}_2$, (DC), if
there exists a monotonically decreasing sequence $\Theta(j)\to 0$ such that
for all $\varphi_1\in\mathcal{B}_1$ and all $\varphi_2 \in\mathcal{B}_2$:
\[
\left|\int \varphi_1 \cdot \varphi_2\circ f^j d \nu -\int \varphi_1 d \nu \int \varphi_2 d\nu\right|
\leq
\Theta(j) \|\phi_1\|_{\mathcal{B}_1} \|\phi_2\|_{\mathcal{B}_2},
\]
where $\|\cdot\|_{\mathcal{B}}$ denotes the norm in space $\mathcal{B}$.

If this condition holds and $\B_2$ contains indicator functions on balls with norm of order the measure of the ball, and $\Theta(j)$ decays fast enough then it is easy to prove $\mathcal{D}_r(u^{(k)}_n)$.  When this doesn't happen, there are approximation arguments (eg. \cite{Gupta}) to derive the same result.

In checking condition $\mathcal{D}_r'(u^{(k)}_n)$ the form of the observable is of significance, and in most applications observables take the form $\phi(x)=\psi(\dist(x,\tilde{x}))$ with $\dist(\cdot,\cdot)$ the Euclidean metric. 
Checking $\mathcal{D}_r'(u^{(k)}_n)$ is then reduced to understanding the recurrence statistics for asymptotically
shrinking balls. For other types of observables with general level set geometries, checking $\mathcal{D}_r'(u^{(k)}_n)$
(and its variant) becomes a much harder problem.

We will list a set of dynamical systems to which our results apply, where the details of how to check our mixing conditions are essentially the same as in the references given.

\subsection{Expanding interval maps}
The simplest example is the tent map $f(x)=1-|1-2x|$, defined for $x\in[0,1]$. 
In this case $\nu=\mathrm{Leb}$, and the system has exponentially
decaying $\Theta(j)$ as defined in condition (DC) for 
$\mathcal{B}_1=\mathrm{Lip}$ and $\mathcal{B}_2=L^{\infty}$.
The much more general case of Rychlik maps was considered in \cite{FFT1}: here $\mathcal{B}_1=BV$ and $\mathcal{B}_2=L^1$.

\subsection{Non-uniformly expanding maps: rapid mixing}
Consider the quadratic map $f(x)=ax(1-x)$, defined for $x\in[0,1]$ and parameter set $a\in[0,4]$.  For a positive measure set of parameters it is known that $\nu$ is absolutely continuous with respect to Lebesgue
measure and condition (DC) 
holds for $\Theta(n)=\theta^{n}_0$, (some $\theta_0<1$). Our conditions can be shown to hold using ideas from \cite{Collet}.

\subsection{Intermittent maps with polynomial decay of correlations}
Consider the class of intermittent 
type maps $f:[0,1]\to[0,1]$ which take the form
\begin{equation*}\label{eq.intermittent}
f(x) =
\begin{cases}
x(1 + 2^{\alpha} x^{\alpha}) &\text{if } 0 \le x < \frac{1}{2};\\
2x-1 &\text{if } \frac{1}{2} \le x \le 1,
\end{cases}
\end{equation*}
with $\alpha\in(0,1)$. This system admits an absolutely continuous invariant measure $\nu$ and condition (DC) applies, 
with $\Theta(n)=O\left(n^{1-\frac{1}{\alpha}}\right)$, see \cite{Young2}. 
Our conditions hold as in \cite{HNT}, namely there is an explicit value $\alpha_0\in(0,1)$ such that conditions
$\mathcal{D}_r(u^{(k)}_n)$ and $\mathcal{D}_r'(u^{(k)}_n)$ hold for all $\alpha\in [0, \alpha_0)$. 
We remark that it is the methodology used in the actual checking of the conditions that leads to the permissible range of 
$\alpha$, and the bound $\alpha_0$. It is a conjectural on whether we can take $\alpha_0=1$, see Section \ref{sec.discussion}.

\section{Generating extremal and jump processes from point processes}\label{sec.extreme_process}
Our approach to proving Theorems \ref{thm.Y-process},
\ref{thm.Y-inv.process} and \ref{thm.records} is to consider
weak convergence of point processes $\xi_n$
(defined on subsets of the plane) and apply the CMT.
As before we suppose that there exist constants $a_n, b_n$ and a non-degenerate function
$G$ such that $n\mu\{X_1>u/a_n+b_n\}\to -\log G(u)$,
and therefore have in mind a dynamical system $(\XX,f, \mu)$ with
absolutely continuous invariant measure $\mu$, and
observable
function $\phi(x)=\psi(\dist(x,\tilde{x}))$, with 
$\psi:\mathbb{R}^+\to\mathbb{R}$ regularly (or slowly) varying. 
The key tool is the following theorem.
\begin{thm}\label{thm.plane}
Suppose that $(\XX,f,\mu)$ is a measure preserving dynamical system, 
and suppose that there exist sequences $\{a_n\}, \{b_n\}$ and a non-degenerate
function $G(u)$ such that \eqref{eq.G_rep} holds.
Suppose that for any $r\geq 1$, and any sequence $\{ x_i,i\leq r\}$ condition 
$\mathcal{D}_r(u^{(k)}_n)$ holds together with $\mathcal{D}_r'(u^{(k)}_n)$. 
Then the point process $\xi_n=\sum_{i=1}^{n}\delta_{z(i,n)}$
with $z(i,n)=(\frac{i}{n},a_n(X_i-b_n))$ 
converges weakly to a PRM in state space $E=[0,\infty)\times\mathbb{R}$
whose intensity measure $\lambda=\mathrm{Leb}\times\lambda_G$ with 
$\lambda_{G}([a,b])=\log G(b)-\log G(a)$.
\end{thm} 
We shall postpone the proof of this theorem to Section 
\ref{sec.proof.thm.plane} as it relies heavily on conditions 
$\mathcal{D}_r(u^{(k)}_n)$ and $\mathcal{D}_r'(u^{(k)}_n)$. We now show how
the main theorems follow from Theorem \ref{thm.plane} and the CMT. 
Our approach follows that of \cite{Res_weak, Resnick3} and an outline is as follows.
Let $M$ and $M'$ be two metric spaces with associated metrics $d$ and $d'$ (resp.).
A main difficulty is showing when a given map $h:M\to M'$ is a.s.\ continuous.
An element $\xi\in M$ is a discontinuity of $h$ if there exists a sequence $\{\xi_n\}$ with
$d(\xi_n,\xi)\to 0$ but $d'(h(\xi_n),h(\xi))\not\to 0$. 
The map is a.s.\ continuous if the set
of discontinuities $\xi\in M$ has $\mu$-measure 0. This agrees with the usual notion of 
(dis)continuity, although we must keep track of the underlying metrics being used. 
For definiteness, consider the case where $M=M_p(E)$ for some $E\subset\mathbb{R}^m$, $(m\geq 1)$,
and $M'=\mathbb{D}(0,\infty)$. We take $d$ to be the vague metric and take $d'$ to 
be the Skorokhod $J_1$ metric. To show a.s.\ continuity of $h$ , 
we must show that $d'(h(\xi_n),h(\xi))\to 0$ for all sequences $\xi_n$ with
$\xi_n\overset{v}{\to}\xi$, and for $\mu$-a.e. $\xi\in M$. Applying the CMT will then imply
that $h(\xi_n)\overset{d}{\to}h(\xi)$ for all such sequences $\xi_n\overset{d}{\to}\xi$. 

\begin{proof}[Proof of Theorem \ref{thm.Y-process}] Given the planar point process $\xi:=\sum_{i=1}^{\infty}\delta_{(t_i,y_i)}$, 
let $H_1:M_p((0,\infty)\times\mathbb{R})\to \mathbb{D}(0,\infty)$ 
be the real valued function defined
by $H_1(\xi)(t)=\sup\{y_i: t_i\leq t\}.$ Then $H_1$ maps point processes 
to $\mathbb{D}(0,\infty)$. If $\xi$ is a PRM then as is shown in \cite[Chapter 4]{Resnick3}
the map $H_1$ is a.s.\ continuous with respect to $\xi$. This is achieved by taking 
any sequence $\{\xi_n\}$ with $\xi_n\overset{v}{\to}\xi$ (i.e. converging vaguely) and showing that 
$d_{a,b}(h(\xi_n)(t),h(\xi)(t))\to 0$ for any $0<a<b<\infty$. Checking the latter is sufficient to prove convergence in
$\mathbb{D}(0,\infty)$.

In the case where $\xi$ is a PRM 
with intensity measure $\lambda=\mathrm{Leb}\times\lambda_G$,
(where $\lambda_{G}([a,b])=\log G(b)-\log G(a)$), then
$H_1(\xi)(t)$ has finite dimensional distributions
which coincide with that of an extremal-G process $Y_G(t)$.
Hence for the process $\xi_n$ defined in Theorem \ref{thm.plane},
the CMT asserts that
$H_1(\xi_n)(t)\overset{d}{\to} H_1(\xi)(t)$, and so $Y_n(t)$ converges weakly
to an extremal-G process $Y_{G}(t)$.
\end{proof}

\begin{proof}[Proof of Theorem \ref{thm.Y-inv.process}] 
Given a planar point process $\xi:=\sum_{i=1}^{\infty}\delta_{(t_i,y_i)}$
consider the function $H_2$ defined by $H_2(\xi)(t)=\inf\{t_i:y_i>t\}$. 
This function is again a.s.\ continuous with respect to a PRM $\xi$ \cite[Chapter 4]{Resnick3}. 
If in particular $\xi$ is a PRM 
with intensity measure $\lambda=\mathrm{Leb}\times\lambda_G$,
(and $\lambda_{G}([a,b])=\log G(b)-\log G(a)$), then
$H_2(\xi)(t)$ has finite dimensional distributions equivalent
to those of $Y^{\leftarrow}_G(t)$ (the inverse of $Y_G(t)$).
Hence for the process $\xi_n$ defined in Theorem \ref{thm.plane},
the CMT asserts that
$H_2(\xi_n)(t)\overset{d}{\to} H_2(\xi)(t)$, and so $Y^{\leftarrow}_n(t)$ converges weakly to $Y^{\leftarrow}_G(t)$.

We remark that it is tempting to apply $H_1$ and then a mapping $\tilde{H}$
with $\tilde{H}(y)=y^{\leftarrow}$ for $y\in\mathbb{D}(0,\infty)$. However this latter map is not continuous on 
$\mathbb{D}(0,\infty)$.
\end{proof}

\begin{proof}[Proof of Theorem \ref{thm.records}]
We first consider item (1), and the process $\mathcal{R}_n$.
Consider the subset $\widetilde{\mathbb{D}}(0,\infty)$ 
of $\mathbb{D}(0,\infty)$ consisting of functions which are
constant between isolated jumps (i.e., the jumps do not accumulate anywhere in $(0, \infty)$). 
For an element $Y(t)\in\widetilde{\mathbb{D}}(0,\infty)$, 
let $H_3:\widetilde{\mathbb{D}}(0,\infty)\to M_p(0,\infty)$ be the 
counting function: $H_3(Y(t))=\sum_i\delta_{t_i}(0, t)$, 
where $t_i$ are jump times for $Y(t)$. As shown in \cite{Resnick3}
the function $H_3$ is a.s.\ continuous when restricted to functions on $\widetilde{\mathbb{D}}(0,\infty)$. 
This is achieved by taking a sequence $y_n\in\widetilde{\mathbb{D}}(0,\infty)$ converging to 
$y\in\widetilde{\mathbb{D}}(0,\infty)$ (with respect to $J_1$ metric), and then showing that
$H_3(y_n)\to H_3(y)$ in the vague metric on $M_p(0,\infty)$. The basic observation here 
is that closeness of the graphs of $y_n$ and $y$ (in the $J_1$ sense) implies that their discontinuities are close.
Hence the corresponding point masses of $H_3(y_n)$ and $H(y)$ are close (in the vague metric sense).

If $Y_G(t)$ is an extremal-G process, then $H_3(Y(t))$
is a PRM on $(0,\infty)$ with intensity $\gamma(t)=1/t$.
Hence to get the required convergence result for $\mathcal{R}_n$ 
we apply the composition $H_3\circ H_1$ to the planar point process and then the CMT.

Now consider item (2), and the process $\mathcal{V}_n$.
As in the proof of item (1), we again consider
the function $H_3$, but this time apply it to elements
of $Y^{\leftarrow}_G(t)\in\widetilde{\mathbb{D}}(-\infty,\infty)$.
The corresponding process
$H_3(Y^{\leftarrow}_G)(t)=\sum_{i}\delta_{Y(t_i)}$ is a PRM
with mean-measure $\lambda([a,b])=-\log(-\log G(b))+\log(-\log G(a))$.
Hence to get the required convergence result for $\mathcal{V}_n$, we apply
the composition $H_3\circ H_2$ to the planar point process sequence $\{\xi_n\}$.
This composition is a.s.\ continuous with respect to the PRM $\xi$, and hence we apply
the CMT.
\end{proof}

\section{Discussion}\label{sec.discussion}
In this article we have developed a general approach to prove convergence to extremal processes
for chaotic dynamical systems. We have also established consequential results such as determining the
statistics of record events. We now make several remarks about the wider applicability of our results and scope for
future investigations. Firstly our results apply to dynamical systems that satisfy the 
$\mathcal{D}_r(u^{(k)}_n)$ and $\mathcal{D}_r'(u^{(k)}_n)$ conditions. For a wide class of non-uniformly expanding dynamical systems, such as those considered within \cite{Collet, FFT1, HNT} our results apply. We note that these conditions
are sufficient for our results. If these conditions fail to hold then it is possible that the conclusions of our results still hold. This might be the case for certain non-mixing systems, such as suspended flows considered in \cite{HNT}. 
Furthermore, we might ask on whether it is possible to by-pass the checking of conditions $\mathcal{D}_r(u^{(k)}_n)$ and 
$\mathcal{D}_r'(u^{(k)}_n)$ to obtain our convergence results. An assumption on the existence of an extreme distribution such as \eqref{eq.Mn-limit} is still expected to be required. However, our results would then apply to a broader class of systems where a link between extremes and return time statistics is known, see \cite{FFT1}. 
 
For hyperbolic systems with attractors (i.e. those that support Sinai-Ruelle-Bowen (SRB) measures) then we 
expect similar conclusions to hold on the existence of an extremal process. For related results on extremes for
hyperbolic systems, see \cite{CC, GHN, Licheng}. To study extremal processes and records for these systems further work is 
required. Particular issues include the actual checking of the conditions $\mathcal{D}_r(u^{(k)}_n)$ and 
$\mathcal{D}_r'(u^{(k)}_n)$, and controlling the regularity of the SRB measure $\mu$ (or regularity of the observable function) 
to ensure existence of a non-degenerate limit function $G$.

It is possible to investigate further situations where $\mathcal{D}_r'(u^{(k)}_n)$ fails, for example
in the case of observable functions maximised at periodic points. For such observables the limit function
$G$ in \eqref{eq.Mn-limit} incorporates an extra parameter known as an \emph{extremal index}, and for dynamical
systems this has been recently studied, for example in \cite{FFT_EI}. To ensure convergence to a corresponding
extremal process, alternative conditions to $\mathcal{D}_r'(u^{(k)}_n)$ would need to be formulated.

The results we have stated about extremal processes and records are not exhaustive. 
Combining Theorem \ref{thm.plane} with the CMT
we can obtain results about the distribution of inter-record times (i.e. $t_{i+1}-t_i$), the jump sizes 
($Y(t_{i+1})-Y(t_{i})$), and the distributions governing the maximum inter-record times/jump sizes, see 
\cite{Res_weak,Res_inv, Resnick3}. For systems that satisfy conditions $\mathcal{D}_r(u^{(k)}_n)$ and
$\mathcal{D}_r'(u^{(k)}_n)$ the joint asymptotic distribution of maxima can also be derived, see 
\cite[Section 5.6]{Leadbetter}. In other directions beyond the scope of this work, we should mention that planar 
convergence of point processes to a PRM (i.e. conclusions similar to that of Theorem \ref{thm.plane}) have been used 
in the study of convergence to Levy processes for certain dynamical systems, see \cite{Ty}. 

Finally we note that our results are all on distributional convergence. In the i.i.d case almost sure convergence results for records is known. 
For an i.i.d process $\{X_n\}$, let $\tau_n$ denote the time to the $n$'th record (as in \eqref{eq.time-to-record}), and $W_n$ the 
number or records observed up to time $n$. Then almost surely we have $(\tau_n)^{1/n}\to e$ and $W_n/\log n\to 1$. For dynamical
systems we conjecture that similar results hold. For the maximum
process, recent work on almost sure convergence is established in \cite{HNT2}. 
In the i.i.d case, almost sure convergence for records is proved by embedding
the maximum process $M_n$ into the extremal process $Y(t)$ for 
$t\in\mathbb{N}$. However, for dependent processes driven by
dynamical systems new ideas are required.

\section{Proof of Theorem \ref{thm.plane}}\label{sec.proof.thm.plane}

It remains to prove Theorem \ref{thm.plane} and we do this as follows.
In the first instance we show in Proposition
\ref{prop.dr_un} how conditions $\mathcal{D}_r(u^{(k)}_n)$ 
and $D_r(u_n)$ can be used to obtain asymptotic independence of blocks 
of maxima on disjoint intervals. Using this result, we then apply the criteria
of Kallenberg to prove the theorem via a thinning
construction as used in \cite{Leadbetter}. 

\subsection{Asymptotic independence of maxima on 
disjoint intervals}

We need to show that the behaviour of maxima in disjoint intervals
is approximately independent. For dynamical systems, condition
$D_r(u^{(k)}_n)$ as used in \cite{Leadbetter} is not readily verifiable, and
hence we propose condition $\mathcal{D}_r(u^{(k)}_n)$ instead, inspired by \cite{FFT1}. As
mentioned in Section \ref{sec.applications}, such a condition can be
easily checked for dynamical systems once information on decay of
correlations is known. However we must show that this condition leads 
to the same conclusion in 
order to apply the thinning constructions of \cite{Leadbetter}, which is the conclusion of the following result.

\begin{prop}\label{prop.dr_un}
Suppose that $p\in\N$ and $A=\cap_{j=1}^{p}I_j$, where $I_j=[a_j,b_j)$. Let $ x_k>0$ for $k=1,\ldots, p$ and for each $1\leq k\leq p$, $(u_{n}^{(k)})_n$ be such that
$n\mu(X_0>u_{n}^{(k)})\to x_k$. Assume that conditions $\mathcal{D}_p'(u_n^{(k)})$ and 
${\mathcal{D}}_p(u_n^{(k)})$ hold. Then
$$\mu\left(\bigcap_{j=1}^{p}\{M(nI_j)\leq u_{n}^{(j)}\}\right)\to\prod_{j=1}^{p} e^{- x_j(b_j-a_j)}.$$
\end{prop}
The proof of this proposition is technical and extends the ideas presented 
in \cite[Section 4]{FFT1}. We postpone the proof until Section \ref{sec.proof.prop.dr_un}.

\subsection{Thinning constructions for point processes.}
As a first step we consider the notion of an independent thinning of a 
Poisson point process $\xi$, see \cite[Section 5.5]{Leadbetter}. We say that a process
$\hat{\xi}$ is an independent thinning of $\xi$ with parameter $p\in(0,1)$, if for every point of $\xi$, there is a probability
$p$ that this point is retained in $\hat{\xi}$. 
An elementary argument shows that if $\xi$ has intensity 1, then $\hat{\xi}$ has intensity $p$.

Suppose now that the sequence $ x_1<\cdots< x_r$ is defined, and $u^{(k)}_{n}$ is such that
$n\mu\{X_1>u^{(k)}_n\}\to x_k$. Then $u^{(1)}_n>\cdots> u^{(r)}_n$. Fix horizontal lines $L_1,\ldots, L_r$
in the plane, and define $\delta^{(k)}_{j/n}$ to be the Dirac mass concentrated at the point on $L_k$ whose
horizontal coordinate is $j/n$. The actual position of each $L_k$ is not important provided their (vertical) order is preserved as described below. Given $k\leq r$, we let $\xi^{(k)}_n$ denote the point process
\begin{equation*}\label{eq.pt-line}
\xi^{(k)}_n=\sum_{j=1}^{n}\delta^{(k)}_{j/n}\cdot1_{\{X_j>u^{(k)}_n\}},
\end{equation*}
and let 
$\tilde{\xi}_n=\sum_{k=1}^{r}\xi^{(k)}_n.$ Notice that for each $k$, if a $\tilde{\xi}^{(k)}_n$ has a point at location
$j/n$, then each $\xi^{(l)}_n$ has a points at $j/n$ too (for all $l\geq k$). 
This follows by the ordering of
$u^{(k)}_n$. 

We next define the point process $\tilde{\xi}=\sum_{k=1}^{r}\xi^{(k)}$, 
where $\{\xi^{(k)}\}$ is a sequence of (independently) thinned Poisson point processes, each having points on respective lines $L_k$, with $L_1>L_{2}>\cdots>L_r$ (ordered in vertical height), and with corresponding intensity parameter $ x_k$. In particular, for each $k$, $\xi^{(k)}$ is an independent thinning of 
$\xi^{(k+1)}$ with probability parameter $p= x_k/ x_{k+1}$. 
We have the following result:
\begin{prop}\label{prop.lines}
Suppose that $(\XX, f, \mu)$ is a measure preserving system. Suppose for any $r\geq 1$, and any sequence
$\{ x_i,i\leq r\}$ condition $\mathcal{D}_r(u^{(k)}_n)$ holds together with
$\mathcal{D}_r'(u^{(k)}_n)$. Then the point process $\tilde{\xi}_n$ defined on lines $\{L_k\}_{k=1}^{r}$ 
(as described above) converges weakly to the point process $\tilde{\xi}$ on $(0,1]\times\mathbb{R}$.
\end{prop}
\begin{rmk}
For ease of exposition we prove convergence on $(0,1]\times\mathbb{R}$. The method of proof extends to
versions converging on state space $(0,\infty)\times\mathbb{R}$. In this case we require
the $\mathcal{D}_r(u^{(k)}_n)$ and $\mathcal{D}_r'(u^{(k)}_n)$ conditions to hold for sequences of the form
$\{u^{(k)}_n(mx_k)\}$, (for all $m\geq 1$). Here $\{u^{(k)}_n(x_k)\}$ is the sequence defined in \eqref{eq.u-kn}.
\end{rmk}
\begin{proof}[Proof of Proposition~\ref{prop.lines}]
Using the criteria of Kallenberg it is sufficient to check the following:
\begin{enumerate}
\item[(a)] $\E(\tilde{\xi}_n(B))\to \E(\tilde{\xi}(B))$, 
where $B$ is of the form $(a,b]\times(c,d]$,
for $c<d$, $0<a<b$, and $b\leq 1$.
\item[(b)] $\mu\{\tilde{\xi}_n=0\}\to\mu\{\tilde{\xi}=0\}$ 
for all $B=\cup_j (A_j\times C_j)$, where $A_j$, $C_j$
are semi-closed intervals as described in (a) above, and with $A_j\times C_j$ disjoint.
\end{enumerate}
Checking these conditions follows \cite[Section 5.5]{Leadbetter}
as applied to dependent processes. In particular
checking part (a) is straightforward and is done as follows. We suppose that $B:=(a,b]\times(c,d]$ intersects lines $L_s,\ldots, L_t$
for $1\leq s\leq t\leq r$. Then
$\tilde{\xi}_n(B)=\sum_{k=s}^{t}\xi^{(k)}_n((a,b])$ and $\tilde{\xi}(B)=\sum_{k=s}^{t}\xi^{(k)}((a,b])$. The expectation of the latter
is $(b-a)\sum_{k=t}^{s} x_k$. Taking expectations of the former we obtain:
\begin{equation*}
\begin{split}
\E(\tilde{\xi}_n(B)) &=\sum_{k=1}^{r}E\left(\sum_{j=1}^{n}\delta^{(k)}_{j/n}(B)\cdot1_{\{X_j>u^{(k)}_n\}}  \right)
=([nb]-[na])\sum_{k=s}^t\mu\{X_1>u^{(k)}_n\}\\
&\sim n(b-a)\sum_{k=s}^{t}\frac{ x_k}{n}\to \E(\tilde{\xi}(B)).
\end{split}
\end{equation*}
We now check part (b) where Proposition \ref{prop.dr_un} is used.
Write $B=\cup_j (A_j\times C_j)$, with $A_j=(a_j,b_j]$, and $C_j=(c_j,d_j]$. By a rearrangement,
we can express $B$ as $\cup_j (a_j, b_j]\times D_j$, with $(a_j,b_j]$ disjoint, and $D_j$ a finite
union of semi-closed intervals. Hence
$$\{\tilde{\xi}_n(B)=0\}=\bigcap_j\{\tilde{\xi}_n(E_j)=0\},$$
where $E_j=(a_j,b_j]\times D_j$. For each $j$, we take $L_{k_j}$ to be the lowest line intersecting $E_j$. By
the thinning construction of $\tilde{\xi}_n$,
$$\{\tilde{\xi}_n(E_j)=0\}=\{\xi^{(k_j)}_n((a_j,b_j])=0\}.$$
This corresponds to the set $\{M(([a_jn],[b_jn]))\leq u_{k_j}\}.$ We can now 
immediately apply Proposition \ref{prop.dr_un} to conclude that
\begin{equation*}
\mu\{\tilde{\xi}_n(B)=0\}\to\exp\left\{-\sum_j(b_j-a_j) x_{k_j}\right\}=\mu\{\tilde{\xi}(B)=0\}.
\end{equation*} 
\end{proof}

\subsection{Concluding the proof of Theorem \ref{thm.plane}}
Given Proposition \ref{prop.lines}, we now show that the process $\xi_n=\sum_{i\geq 1}\delta_{z(i,n)}$
converges to a Poisson process on the plane with intensity measure $\lambda=\mathrm{Leb}\times\lambda_G$,
with $\lambda_G[a,b]=\log G(b)-\log G(a)$. The argument is purely probabilistic and follows \cite[Section 5.7]{Leadbetter}. 
We give the main steps. It is convenient to work with the process $\hat{\xi}_n=\sum_{i\geq 1}\delta_{w(i,n)}$
with $w(i,n)=\left(\frac{i}{n},u^{-1}_n(X_i)\right)$. Recall that the function $u_n(x)$ is determined 
via the limit relation $n\mu\{X_1>u_n(x)\}\to x$ (from \eqref{eq.u-kn}), and hence if \eqref{eq.un_rep} applies
then $w(i,n)=\left(\frac{i}{n},-\log G(a_n(X_i-b_n))\right)$. We show that $\hat\xi_n$ converges to a Poisson process 
$\hat{\xi}$ in state space $(0,\infty)\times(0,\infty)$, with Lebesgue as the intensity measure. A simple change of measure
argument then shows that $\xi$ is a PRM with the corresponding intensity measure $\lambda$.

Continuing with the proof, and by Kallenberg's criteria it is sufficient to check items 
(a), and (b) as specified in the proof of Proposition \ref{prop.lines} (for a suitable collection of disjoint 
semi-closed rectangles). If $B=(a,b]\times[c, d)$, then
\begin{equation*}
\begin{split}
\E(\hat{\xi}_n(B))&=E\left(\sum_{j\geq 1}\delta_{w(j,n)}(B)\right)
=([nb]-[na])\mu\{c\leq u^{-1}_n(X_1)<d\}\\
&\sim n(b-a)\mu\{u_n(c)\leq X_1\leq u_n(d)\}\sim (b-a)(d-c)=\E(\hat{\xi}(B)).
\end{split}
\end{equation*}
To show (b), we consider the event $\{\hat{\xi}_n(B)=0\}$, where (as before) $B=\cup_j (A_j\times C_j)$, with $A_j=(a_j,b_j]$, $C_j=[c_j,d_j)$, and $A_j\times C_j$ forming a disjoint collection. By a rearrangement,
we can express $B$ as $\cup_j (a_j, b_j]\times D_j$, with $(a_j,d_j]$ disjoint, and $D_j$ a finite
union of semi-closed intervals.  Indeed we can also make the simplifying assumption that $B=\cup_j\left( A\times D_j\right)$ for $A=(a,b]$, since the proof essentially is the same otherwise. 
So writing $B=\bigcup_{j=1}^{m}\left(A\times D_j\right)$, with $D_j=[ x_{2j-1}, x_{2j}]$, for all $j\leq m$, and
some sequence $ x_1<\cdots< x_{2m}$,
\begin{equation}\label{eq.noentry}
\{\hat{\xi}_n(B)=0\}=
\bigcap_{j=1}^{m}\left\{\xi^{(2j-1)}_n(A)=\xi^{(2j)}_n(A)\right\},
\end{equation}
where $\xi^{(j)}_n(A)=\sharp\{\ell\leq n:\,\ell/n\in A,\,u_n(x_j)<X_{\ell}\}$, i.e. $\xi^{(j)}$
counts the number of times of an exceedance $\ell\le n$ of $u_n(x_j)$ where $\ell/n\in A$.
The decomposition in \eqref{eq.noentry} corresponds to the fact that we cannot have any
$X_{\ell}$ with $u_n(x_{2j})<X_{\ell}<u_n(x_{2j-1})$, and therefore an exceedance of $u_n(x_{2j})$ implies
an exceedance of $u_n(x_{2j-1})$.

The processes $\xi^{(j)}_n$ meet the criteria of the thinning processes used in Proposition \ref{prop.lines}, and so:
$$(\xi^{(1)}_n(A),\ldots,\xi^{(2m)}_n(A))\to
(\xi^{(1)}(A),\ldots,\xi^{(2m)}(A)),$$
where convergence is in distribution. The distributions $\xi^{(j)}$ ($1\leq j\leq 2m$) correspond to the successively thinned Poisson processes defined in Proposition \ref{prop.lines}, and we obtain
\begin{equation}\label{eq.reduction}
\mu\{\hat{\xi}_n(B)=0\}\to\mu\left(\bigcap_{j=1}^{m}\{\xi^{(2j-1)}(A)=\xi^{2j}(A)\} \right).
\end{equation}
To compute the right hand side of \eqref{eq.reduction}, we use the thinning properties
of the $\xi^{(j)}$. Note first of all that
\begin{equation*}
\begin{split}
\mu\left(\xi^{(2m-1)}(A)=\xi^{(2m)}(A) \right) &=
\sum_{k=1}^{\infty}\frac{( x_{2m}(b-a))^ke^{- x_{2m}(b-a)}}{k!}\left(\frac{ x_{2m-1}}{ x_{2m}}\right)^k\\
&=e^{-( x_{2m}- x_{2m-1})(b-a)},
\end{split}
\end{equation*}
and by successive thinning we obtain for $j_{m-1}\leq j_m$:
\begin{multline*}
\mu\left\{\left(\xi^{(2m-1)}(A)=\xi^{(2m)}(A)=j_m \right)\cap
\left(\xi^{(2m-3)}(A)=\xi^{(2m-3)}(A)=j_{m-1} \right)\right\}
=\\
\frac{( x_{2m}(b-a))^{j_m}e^{- x_{2m}(b-a)}}{j_m!}
{{j_m}\choose{j_{m-1}}}\left(\frac{ x_{2m-3}}{ x_{2m-1}}\right)^{j_{m-1}}\left(1-\frac{ x_{2m-2}}{ x_{2m-1}} 
\right)^{j_m-j_{m-1}}
\end{multline*}
Summing over all $0\leq j_{m-1}\leq j_{m}<\infty$, we obtain
\begin{multline*}
\mu\left\{\left(\xi^{(2m-1)}(A)=\xi^{(2m)}(A)\right)\cap
\xi^{(2m-3)}(A)=\xi^{(2m-3)}(A)\right)=\\
\exp\{-(b-a)( x_{2m}- x_{2m-1}+ x_{2m-2}- x_{2m-3})\}
\end{multline*}
We can clearly iterate this, and a formula for the general case is given by:
\begin{multline}\label{eq.reduction2}
\mu\left\{\bigcap_{k=1}^{m}\left(\xi^{(2k-1)}(A)=\xi^{(2k)}(A)=j_k \right)\right\}=\\
\frac{( x_{2m-1}(b-a))^{j_m}e^{- x_{2m}(b-a)}}{k!}\prod_{k=2}^{m}
{{j_k}\choose{j_{k-1}}}\left(\frac{ x_{2k-3}}{ x_{2k-1}}\right)^{j_{k-1}}
\left(1-\frac{ x_{2k-2}}{ x_{2k-1}} 
\right)^{j_k-j_{k-1}}
\end{multline} 
The probability of the event in \eqref{eq.reduction} is then obtained by summing together
all probabilities in \eqref{eq.reduction2} with $0\leq j_1\leq j_2\leq\ldots\leq j_m$. By 
an iterative method we obtain
$$\mu\left(\bigcap_{j=1}^{m}\{\xi^{(2j-1)}(A)=\xi^{2j}(A)\} \right)
=\exp\{-\sum_{j=1}^{m}( x_{2j}- x_{2j-1})(b-a)\}=\exp\{-\mathrm{Leb}(B)\}.$$
This completes the proof.

\section{Proof of Proposition \ref{prop.dr_un}}\label{sec.proof.prop.dr_un}
We must first adapt  \cite[Lemma 4.2]{FFT1}:
\begin{lemma}
\begin{enumerate}
\item[(a)] Given sets $A_1, \ldots, A_w\subset [0, \infty)$ and $B_1\supset A_1$, let $\ell:=\#\{j\in \N:j\in B_1\sm A_1\}$ and $(u_i)_{i=1}^w\subset \R$,
\begin{align*}&\left|\mu\left(\bigcap_{i=1}^w\{M(A_i)\le u_i\}\right)-\mu\left( \{M(B_1)\le u_1\}\cap \bigcap_{i=2}^w \{M(A_i)\le u_i\}\right)\right|\\
&\hspace{2cm}\le \mu(M(A_1)\le u_1)-\mu(M(B_1)\le u_1)\le\ell \mu(X>u_1),
\end{align*}

\item[(b)] For $w\in \N$, assuming $\min\{x:x\in A_i, i=1, \ldots, w\}\ge r+t$, for $A_0=[0, r+t)$, for $u'>0$,
\begin{align*}&\Bigg|\mu\left(\{M(A_0)\le u'\}\cap\bigcap_{j=1}^w\{M(A_j)\le u_j\}\right)-\mu\left(\bigcap_{j=1}^w \{M(A_j)\le u_j\}\right)\\
&\hspace{6cm}+ \sum_{i=0}^{r-1}\mu\left(\{X>u'\}\cap\bigcap_{j=1}^w\{M(A_j-i)\le u_j\}\right)\Bigg|\\
&\hspace{5cm}\le 2r\sum_{i=1}^{r-1}\mu\left(\{X>u'\}\cap\{X_i>u'\}\right)+t\mu(X>u').
\end{align*}
\end{enumerate}
\label{lem:4.2}
\end{lemma}

\begin{proof}
For the first part, the first equality is an elementary argument, while the final inequality follows from (4.1) in \cite[Lemma 4.2]{FFT1}.

The second part follows as in (4.2) of \cite[Lemma 4.2]{FFT1}, itself a minor adaptation of \cite[Lemma 3.2]{FF}.
\end{proof}

\begin{proof}[Proof of Proposition~\ref{prop.dr_un}]
We closely follow the proof of \cite[Proposition 1]{FFT1}.

Let $h:=\inf_{j\in \{1,\ldots,p\}}\{b_j-a_j\}$ and
$H:=\lceil\sup\{x:x\in A\}\rceil$. Take $k>2/h$ and $n$ sufficiently
large. Note this guarantees that if we partition $n[0,H]\cap {\mathbb
Z}$ into blocks of length $r_n:=\lfloor n/k\rfloor$,
$J_1=[Hn-r_n,Hn)$, $J_2=[Hn-2r_n,Hn-r_n)$,\ldots,
$J_{Hk}=[Hn-Hkr_n,n-(Hk-1)r_n)$, $J_{Hk+1}=[0,Hn-Hkr_n)$, then there
is at least one of these blocks contained in $\n I_i$.  Let
$S_\ell=S_\ell(k)$ be the number of blocks $J_j$ contained in $\n
I_\ell$ minus 1, that is,
$$S_\ell:=\#\{j\in \{1,\ldots,Hk\}:J_j\subset \n I_\ell\}-1.$$
So $S_\ell\ge0$ $\forall \ell\in
\{1,\ldots,p\}$. 
Set $i_\ell:=\min\{j\in \{1,\ldots,k\}:J_j\subset \n I_{\ell}\}.$
Then $J_{i_\ell},J_{i_\ell+1},\ldots,J_{i_\ell+S_\ell}\subset n
I_\ell$. Now, fix $\ell$ and for each $ i\in \{i_{p-\ell
+1},\ldots,i_{p-\ell +1}+S_{p-\ell+1}\}$ let
$$
B_i:=\bigcup_{j=i_{p-\ell+1}}^i J_j,\;  J_i^*:=[Hn-ir_n,
Hn-(i-1)r_n-t_n)\; \mbox{ and } J_i':=J_i-J_i^*$$ for $t_n=o(n)$ given in $\tDr(u_n)$. Note that
$|J_i^*|=r_n-t_n$ and $|J_i'|=t_n$.  
See Figure~\ref{fig:notation} for
more of an idea of the notation here. 
\begin{figure}[h]
  \includegraphics[scale=1]{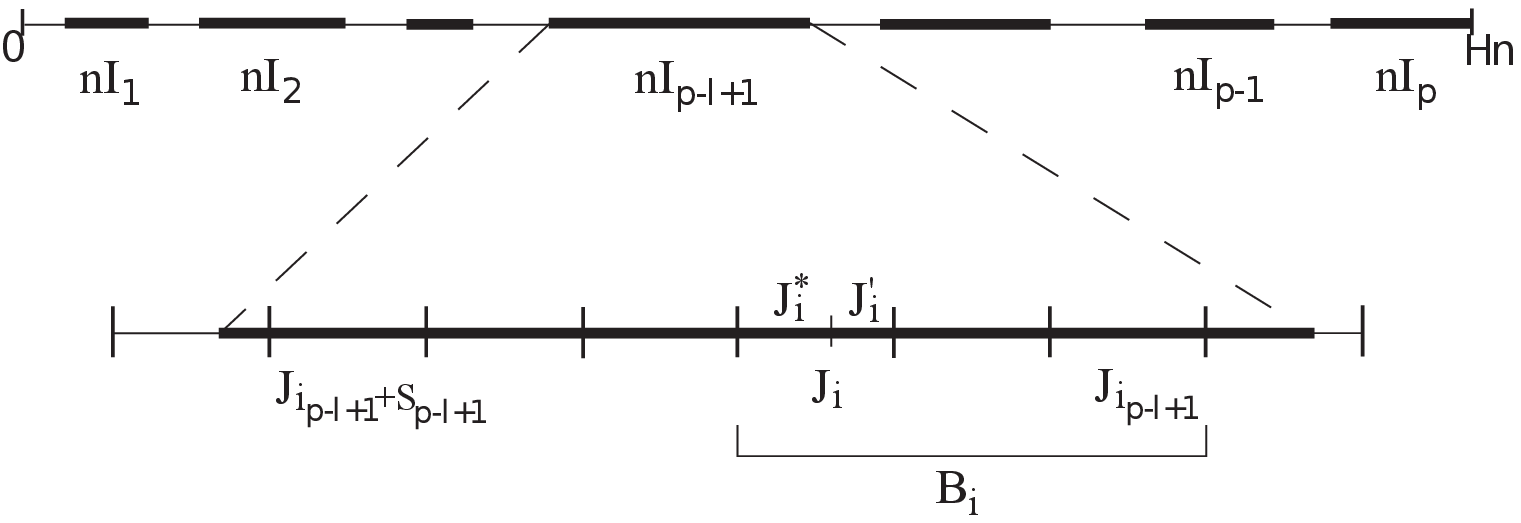}
   \caption{Notation}\label{fig:notation.pdf}
\end{figure}\label{fig:notation}

For the first part of the proof, we write $u_{n}^{(j)}=u_j$.
Then
 for any $u'\in \R$ 

\begin{align*}
&\Bigg|\mu\left(\{M(B_i)\le u'\}\cap\left(\bigcap_{i=1}^{\ell-1} \{M(nI_{p-i+1})\le u_{p-i+1}\}\right) \right)\\
&\hspace{3cm}-(1-r_n\mu(X>u'))\mu\left(\{M(B_{i-1})\le u'\}\cap\left(\bigcap_{i=1}^{\ell-1} \{M(nI_{p-i+1})\le u_{p-i+1}\}\right) \right)\Bigg|\\
&\le \Bigg|\mu\left(\{M(B_i)\le u'\}\cap\left(\bigcap_{i=1}^{\ell-1} \{M(nI_{p-i+1})\le u_{p-i+1}\}\right) \right)\\
&\hspace{2cm}-\mu\left(\{M(B_{i-1})\le u'\}\cap\left(\bigcap_{i=1}^{\ell-1} \{M(nI_{p-i+1})\le u_{p-i+1}\}\right) \right)\\
&\hspace{4cm}+r_n\mu(X>u')\mu\left(\{M(B_{i-1})\le u'\}\cap\left(\bigcap_{i=1}^{\ell-1} \{M(nI_{p-i+1})\le u_{p-i+1}\}\right) \right)\Bigg|\\
&\le \Bigg|\mu\left(\{M(B_i)\le u'\}\cap\left(\bigcap_{i=1}^{\ell-1} \{M(nI_{p-i+1})\le u_{p-i+1}\}\right) \right)\\
&\hspace{1cm}-\mu\left(\{M(B_{i-1})\le u'\}\cap\left(\bigcap_{i=1}^{\ell-1} \{M(nI_{p-i+1})\le u_{p-i+1}\}\right) \right)\\
&\hspace{1cm}+\sum_{j=0}^{r_n-t_n-1}\mu\left(\{X_{j+Hn-ir_n}>u'\}\cap\{M(B_{i-1})\le u'\}\cap \left(\bigcap_{i=1}^{\ell-1} \{M(nI_{p-i+1})\le u_{p-i+1}\}\right)\right)\Bigg|\\
&
\quad+\Bigg|(r_n-t_n)\mu(X>u')\mu\left(\{M(B_{i-1})\le u'\}\cap\left(\bigcap_{i=1}^{\ell-1} \{M(nI_{p-i+1})\le u_{p-i+1}\}\right) \right)\\
&\hspace{1.5cm} -\sum_{j=0}^{r_n-t_n-1}\mu\left(\{X_{j+Hn-ir_n}>u'\}\cap \{M(B_{i-1})\le u'\}\cap \left(\bigcap_{i=1}^{\ell-1} \{M(nI_{p-i+1})\le u_{p-i+1}\}\right)\right)\Bigg|\\
&\hspace{8cm}+t_n\mu(X>u').\\
\end{align*}

\iffalse
\begin{align*}
|\mu(M(B_i\cup \n & A_{\ell-1})\leq
u_n)-(1-r_n\mu(X>u_n))\mu(M(B_{i-1}\cup \n A_{\ell-1})\leq u_n)|\\
&=\Big|\mu(M(B_i\cup \n A_{\ell-1})\leq u_n)-\mu(M(B_{i-1}\cup \n
A_{\ell-1})\leq u_n)\\&\qquad\qquad\quad+r_n \mu(X>
u_n)\mu(M(B_{i-1}\cup \n A_{\ell-1}) \leq u_n)\Big|\\ & \leq
\Big|\mu(M(B_i\cup \n A_{\ell-1})\leq u_n)-\mu(M(B_{i-1}\cup \n
A_{\ell-1})\leq u_n)\\ &\qquad\qquad\quad+(r_n-t_n)\mu(X>
u_n)\mu(M(B_{i-1}\cup \n A_{\ell-1}) \leq u_n)\Big|\\
 &\quad+t_n\mu(X> u_n)\mu(M(B_{i-1}\cup \n A_{\ell-1}) \leq u_n)\\
&\leq \Big|\mu(M(B_i\cup \n A_{\ell-1})\leq u_n)-\mu(M(B_{i-1}\cup \n
A_{\ell-1})\leq u_n)\\
&\qquad\qquad\quad+\sum_{j=0}^{r_n-t_n-1}\mu(\{X_{j+Hn-ir_n}>u_n\}\cap
\{M(B_i\cup \n A_{\ell-1})\leq u_n)\}\Big|\\
&\quad+\Big|(r_n-t_n)\mu(X>u_n)\mu(M(B_{i-1}\cup \n A_{\ell-1})\leq
u_n)\\
&\qquad\qquad\quad-\sum_{j=0}^{r_n-t_n-1}\mu(\{X_{j+Hn-ir_n}>u_n\}\cap
\{M(B_i\cup \n A_{\ell-1})\leq u_n)\}\Big|\\&\quad+t_n\mu(X>u_n).
\end{align*}
\fi

We next apply Lemma~\ref{lem:4.2}(b) to the first sum in absolute value above. 
Since $B_i=J_i\cup B_{i-1}= J_i^*\cup J_i'\cup B_{i-1}$, we can translate the sets in the first two terms here  back by $Hn-ir_n$ and with $J_i-(Hn-ir_n)=[0, r_n)$ taking the place of $A_0$ in our lemma, and correspondingly shifting the terms in the sum there by $d_j=(j+Hn-ir_n)$.  Then (also shifting terms in the second absolute value by $d_j$ and using the invariance of $\mu$),

\begin{align*}
&\Bigg|\mu\left(\{M(B_i)\le u'\}\cap\left(\bigcap_{i=1}^{\ell-1} \{M(nI_{p-i+1})\le u_{p-i+1}\}\right) \right)\\
&\hspace{2cm}-(1-r_n\mu(X>u'))\mu\left(\{M(B_{i-1})\le u'\}\cap\left(\bigcap_{i=1}^{\ell-1} \{M(nI_{p-i+1})\le u_{p-i+1}\}\right) \right)\Bigg|\\
&\hspace{5mm}\le 2(r_n-t_n)\sum_{j=1}^{r_n-t_n-1}\mu(\{X>u'\}\cap\{X_j>u'\})
+t_n\mu(X>u')\\
&\hspace{15mm}+\sum_{j=0}^{r_n-t_n-1}\Bigg|\mu(X>u')\mu\left(\{M(B_{i-1})\le u'\}\cap\left(\bigcap_{i=1}^{\ell-1} \{M(nI_{p-i+1})\le u_{p-i+1}\}\right) \right)\\
&\hspace{2cm} -\mu\left(\{X>u'\}\cap\{M(B_{i-1}-d_j)\le u'\}\cap \left(\bigcap_{i=1}^{\ell-1} \{M(nI_{p-i+1}-d_j)\le u_{p-i+1}\}\right)\right)\Bigg|\\
&\hspace{10cm}+t_n\mu(X>u').
\end{align*}
 Now using condition $\tDr(u_n)$, we obtain

\begin{align*}
&\Bigg|\mu\left(\{M(B_i)\le u'\}\cap\left(\bigcap_{i=1}^{\ell-1} \{M(nI_{p-i+1})\le u_{p-i+1}\}\right) \right)\\
&\hspace{3cm}-(1-r_n\mu(X>u'))\mu\left(\{M(B_{i-1})\le u'\}\cap\left(\bigcap_{i=1}^{\ell-1} \{M(nI_{p-i+1})\le u_{p-i+1}\}\right) \right)\Bigg|\\
& \leq
2(r_n-t_n)\sum_{j=1}^{r_n-t_n-1}\mu(\{X>u'\}\cap\{X_j>u'\})
+2t_n\mu(X>u')+(r_n-t_n)\gamma(n,t_n).
\end{align*}
Set
$$\Upsilon_{k,n}(u'):=2(r_n-t_n)\sum_{j=1}^{r_n-t_n-1}
\mu(\{X>u'\}\cap\{X_j>u'\})
+2t_n\mu(X>u')+(r_n-t_n)\gamma(n,t_n).$$ By the definition of $u_j=u_{n, j}$, we
may assume that $n$ and $k$ are sufficiently large so that
$\frac{n}{k}\mu(X>u_j)<2$ and $|1-r_n\mu(X>u_j)|<1$ which implies
\begin{align*}
&\Bigg|\mu\left(\{M(B_{S_{p-\ell+1}})\le u_{p-\ell+1}\}\cap\left(\bigcap_{i=1}^{\ell-1} \{M(nI_{p-i+1})\le u_{p-i+1}\}\right) \right)\\
&\hspace{0.5cm}-(1-r_n\mu(X>u_{p-\ell+1}))\mu\left(\{M(B_{S_{p-\ell+1}-1})\le u_{p-\ell+1}\}\cap\left(\bigcap_{i=1}^{\ell-1} \{M(nI_{p-i+1})\le u_{p-i+1}\}\right) \right)\Bigg|\\
&\qquad\le\Upsilon_{k,n}(u_{p-\ell+1}),
\end{align*}
and

\begin{align*}
&\Bigg|\mu\left(\{M(B_{S_{p-\ell+1}})\le u_{p-\ell+1}\}\cap\left(\bigcap_{i=1}^{\ell-1} \{M(nI_{p-i+1})\le u_{p-i+1}\}\right) \right)\\
&\hspace{0.3cm}-(1-r_n\mu(X>u_{p-\ell+1}))^2\mu\left(\{M(B_{S_{p-\ell+1}-2})\le u_{p-\ell+1}\}\cap\left(\bigcap_{i=1}^{\ell-1} \{M(nI_{p-i+1})\le u_{p-i+1}\}\right) \right)\Bigg|\\
&\le \Bigg|\mu\left(\{M(B_{S_{p-\ell+1}})\le u_{p-\ell+1}\}\cap\left(\bigcap_{i=1}^{\ell-1} \{M(nI_{p-i+1})\le u_{p-i+1}\}\right) \right)\\
&\hspace{0.5cm}-(1-r_n\mu(X>u_{p-\ell+1}))\mu\left(\{M(B_{S_{p-\ell+1}-1})\le u_{p-\ell+1}\}\cap\left(\bigcap_{i=1}^{\ell-1} \{M(nI_{p-i+1})\le u_{p-i+1}\}\right) \right)\Bigg|\\
&\quad+|1-r_n\mu(X>u_{p-\ell+1})|\Bigg|\mu\left(\{M(B_{S_{p-\ell+1}-1})\le u_{p-\ell+1}\}\cap\left(\bigcap_{i=1}^{\ell-1} \{M(nI_{p-i+1})\le u_{p-i+1}\}\right) \right)\\
&\hspace{0.7cm}-(1-r_n\mu(X>u_{p-\ell+1}))\mu\left(\{M(B_{S_{p-\ell+1}-2})\le u_{p-\ell+1}\}\cap\left(\bigcap_{i=1}^{\ell-1} \{M(nI_{p-i+1})\le u_{p-i+1}\}\right) \right)\Bigg|\\
&\le 2\Upsilon_{k,n}(u_{p-\ell+1}),
\end{align*}

Inductively, we obtain
\begin{align*}
&\Bigg|\mu\left(\{M(B_{S_{p-\ell+1}})\le u_{p-\ell+1}\}\cap\left(\bigcap_{i=1}^{\ell-1} \{M(nI_{p-i+1})\le u_{p-i+1}\}\right) \right)\\
&\hspace{1cm}-(1-r_n\mu(X>u_{p-\ell+1}))^{S_{p-\ell+1}}\mu\left(\bigcap_{i=1}^{\ell-1} \{M(nI_{p-i+1})\le u_{p-i+1}\}\right)\Bigg|\\
&\qquad\le S_{p-\ell+1}\Upsilon_{k,n}(u_{p-\ell+1}).
\end{align*}
Using Lemma~\ref{lem:4.2}(a),
\begin{align*}
&\Bigg|\mu\left(\bigcap_{i=1}^{\ell} \{M(nI_{p-i+1})\le u_{p-i+1}\}\right)\\
&\hspace{1cm}-(1-r_n\mu(X>u_{p-\ell+1}))^{S_{p-\ell+1}}\mu\left(\bigcap_{i=1}^{\ell-1} \{M(nI_{p-i+1})\le u_{p-i+1}\}\right)\Bigg|\\
&\le\Bigg|\mu\left(\bigcap_{i=1}^{\ell} \{M(nI_{p-i+1})\le u_{p-i+1}\}\right) \\
&\hspace{2cm}-\mu\left(\{M(B_{S_{p-\ell+1}})\le u_{p-\ell+1}\}\cap\left(\bigcap_{i=1}^{\ell-1} \{M(nI_{p-i+1})\le u_{p-i+1}\}\right) \right)\Bigg|\\
&\quad+\Bigg|\mu\left(\{M(B_{S_{p-\ell+1}})\le u_{p-\ell+1}\}\cap\left(\bigcap_{i=1}^{\ell-1} \{M(nI_{p-i+1})\le u_{p-i+1}\}\right) \right)\\
&\hspace{1cm}-(1-r_n\mu(X>u_{p-\ell+1}))^{S_{p-\ell+1}}\mu\left(\bigcap_{i=1}^{\ell-1} \{M(nI_{p-i+1})\le u_{p-i+1}\}\right)\Bigg|\\
&\le\Bigg|\mu\left(\bigcap_{i=1}^{\ell} \{M(nI_{p-i+1})\le u_{p-i+1}\}\right) \\
&\hspace{2cm}-\mu\left(\left\{M\left(\cup_{i=i_\ell}^{S_{p-\ell+1}} J_i\right)\le u_{p-\ell+1}\right\}\cap\left(\bigcap_{i=1}^{\ell-1} \{M(nI_{p-i+1})\le u_{p-i+1}\}\right) \right)\Bigg|\\
&\hspace{8cm}+ S_{p-\ell+1}\Upsilon_{k,n}(u_{p-\ell+1})\\
&
\le 2r_n
\mu(X>u_{p-\ell+1})+S_{p-\ell+1}\Upsilon_{k,n}(u_{p-\ell+1}).
\end{align*}
For the next step we estimate 
\begin{align*}
&\Bigg|\mu\left(\bigcap_{i=1}^{\ell} \{M(nI_{p-i+1})\le u_{p-i+1}\}\right) \\
&\hspace{0.3cm}-(1-r_n\mu(X>u_{p-\ell+1}))^{S_{p-\ell+1}}(1-r_n\mu(X>u_{p-\ell+2}))^{S_{p-\ell+2}}\mu\left(\bigcap_{i=1}^{\ell-2} \{M(nI_{p-i+1})\le u_{p-i+1}\}\right)\Bigg|\\
&\le \Bigg|\mu\left(\bigcap_{i=1}^{\ell} \{M(nI_{p-i+1})\le u_{p-i+1}\}\right) \\
&\hspace{3cm}-(1-r_n\mu(X>u_{p-\ell+1}))^{S_{p-\ell+1}}\mu\left(\bigcap_{i=1}^{\ell-1} \{M(nI_{p-i+1})\le u_{p-i+1}\}\right)\Bigg|\\
&\hspace{0.4cm}+|1-r_n\mu(X>u_{p-\ell+2})|^{S_{p-\ell+2}}\Bigg|\mu\left(\bigcap_{i=1}^{\ell-1} \{M(nI_{p-i+1})\le u_{p-i+1}\}\right) \\
&\hspace{3cm}-(1-r_n\mu(X>u_{p-\ell+1}))^{S_{p-\ell+2}}\mu\left(\bigcap_{i=1}^{\ell-2} \{M(nI_{p-i+1})\le u_{p-i+1}\}\right)\Bigg|\\
 &\le 2r_n\left(\mu(X>u_{p-\ell+1})+\mu(X>u_{p-\ell+2})\right)+S_{p-\ell+1}\Upsilon_{k,n}(u_{p-\ell+1})+S_{p-\ell+2}\Upsilon_{k,n}(u_{p-\ell+2})
\end{align*}
Therefore, by induction, we obtain
\begin{align*}
&\Bigg|\mu\left(\bigcap_{i=1}^{\ell} \{M(nI_{p-i+1})\le u_{p-i+1}\}\right) -\prod_{j=1}^p\left(1-r_n\mu(X>u_{j})\right)^{S_{j}}\Bigg|\\
&\hspace{7cm}\le 2r_n\sum_{j=1}^p\mu(X>u_j)+\sum_{j=1}^pS_j\Upsilon_{k,n}(u_j).
\end{align*}
Now, it is easy to see that $S_j\sim k|I_j|$, for each
$j\in\{1,\ldots,p\}$. Consequently, recalling $u_{n, j}=u_j$,
\begin{align*}
&\lim_{k\rightarrow+\infty}\lim_{n\rightarrow+\infty}
\prod_{j=1}^p\left(1-r_n\mu(X>u_{j})\right)^{S_{j}}= \lim_{k\rightarrow+\infty}\lim_{n\rightarrow+\infty}
\prod_{j=1}^p\left(1-\left\lfloor\frac{n}{k}\right\rfloor\mu(X>u_{j})\right)^{S_{j}}\\
&= \lim_{k\rightarrow+\infty}\lim_{n\rightarrow+\infty}
\prod_{j=1}^p\left(1-\left\lfloor\frac{n}{k}\right\rfloor\mu(X>u_{j})\right)^{k|I_{j}|}= \lim_{k\rightarrow+\infty}\lim_{n\rightarrow+\infty}
\prod_{j=1}^p\left(1-\frac{ x_j}{k}\right)^{k|I_{j}|}=
\prod_{j=1}^p e^{- x_j|I_{j}|}\\
\end{align*}

To conclude the proof it suffices to show that
\[
\lim_{k\rightarrow+\infty}\lim_{n\rightarrow+\infty}\left(2r_n\sum_{j=1}^p\mu(X>u_{n,j})+\sum_{j=1}^pS_j\Upsilon_{k,n}(u_{n,j})\right)=0.
\]
We start by noting that, since $n\mu(X>u_{n,j})\to x_j\geq 0$,
\[
\lim_{k\rightarrow+\infty}\lim_{n\rightarrow+\infty}2r_n \mu(X>u_{n,j})=
\lim_{k\rightarrow+\infty} \frac{2 x_j}{k}=0.
\]
Next we need to check that for each $j=1,\ldots, p$
\begin{multline*}
\lim_{k\rightarrow+\infty}\lim_{n\rightarrow+\infty}
2k(r_n-t_n)\sum_{j=1}^{r_n-t_n-1}
\mu(\{X>u_{n,j}\}\cap\{X_j>u_{n, j}\})
+2t_n\mu(X>u_{n, j}) \\+k(r_n-t_n)\gamma(n,t_n)=0.
\end{multline*}

Recall that $t_n=o(n)$ is given by
$\mathcal{D}_p(u_n^{(k)})$. Now, observe that for each $j=1, \ldots, p$ and every $k\in\N$,
we have \( \lim_{n\to\infty}kt_n\mu(X>u_{n}^{(j)})=0 \). Finally, we use
$\mathcal{D}_p(u_n^{(k)})$ and $\mathcal{D}_p'(u_n^{(k)})$ to prove that the two remaining terms also
go to $0$.
\end{proof}

\section*{Acknowledgements}
The authors wish to thank J. Freitas, M. Nicol and D. Terhesiu for 
useful comments and discussions. This research was partially supported by the London Mathematics Society 
(Scheme 4, no. 41126), and both authors thank the Erwin Schr\"odigner Institute (ESI) in Vienna were part of 
this work was carried out. MH wishes to thank the Department of Mathematics, 
University of Houston for hospitality and financial support, and MT thanks Exeter University for their 
hospitality and support.

\end{document}